\documentclass{amsart}

\usepackage{amssymb}
\usepackage{amsfonts}
\usepackage{amsmath}
\usepackage{graphicx}
\usepackage{shadow}
\usepackage{color}
\usepackage[all]{xy}
\usepackage{mathtools}
\usepackage{amsfonts}
\usepackage[pagebackref]{hyperref}
\newtheorem{thm}{Theorem}[section]

\newtheorem{lemma}[thm]{Lemma}
\newtheorem{proposition}[thm]{Proposition}

\theoremstyle{definition}

\newtheorem{example}[thm]{Example}

\theoremstyle{remark}

\DeclareMathOperator{\id}{id}

\DeclareMathOperator{\m}{\frak{m}}

\DeclareMathOperator{\PP}{\mathcal{P}}

\DeclareMathOperator{\Hom}{Hom}
\DeclareMathOperator{\Ext}{Ext}

\DeclareMathOperator{\tor}{tor}
\DeclareMathOperator{\Tor}{Tor}

\def\pd{{\rm pd}}
\def\fd{{\rm fd}}
\def\wfd{{\rm w-fd}}
\def\id{{\rm id}}
\def\gldim{{\rm gl.dim}}
\def\w-wgldim{{\rm w-w.gl.dim}}
\def\wgldim{{\rm w.gl.dim}}
\def\glwdim{$gl.w.$w$-${\rm dim}}
\def\wwpd{{\rm w}.w$-${\rm pd}}
\def\sup{{\rm sup}}

\everymath{\displaystyle\everymath{}}
\begin{document}
	
	\title[A note on weak $w$-projective modules]{A note on weak $w$-projective modules}

\author[R.A.K. Assaad]{Refat Abdelmawla Khaled Assaad}
\address{Department of Mathematics, Faculty of Science, University Moulay Ismail Meknes, Box 11201, Zitoune, Morocco}
\email{refat90@hotmail.com}



	\begin{abstract}
 Let $R$ be a ring. An $R$-module $M$ is said to be a weak $w$-projective module if $\Ext_R^1(M,N)=0$ for all  $N \in \mathcal{P}_{w}^{\dagger_\infty}$ (see, \cite{FLQ}).
 In this paper, we introduce and study some properties of weak $w$-projective modules. And we use these modules to characterize some classical rings, for example, we will prove that a ring $R$ is a $DW$-ring if and only if every weak $w$-projective is projective, $R$ is a Von Neumann regular ring if and only if every FP-projective is weak $w$-projective if and only if every finitely presented $R$-module is weak $w$-projective and $R$ is a $w$-semi-hereditary if and only if every finite type submodule of a free module is weak $w$-projective if and only if every finitely generated ideal of $R$ is a weak $w$-projective.
\end{abstract}
\maketitle
\section{Introduction}
In this paper, all rings are considered commutative with unity and all modules are unital. Let $R$ be a ring and $M$ be an $R$-module. As usual, we use $\pd_R(M)$, $\id_R(M)$ and $\fd_R(M)$ to
denote, respectively, the classical projective dimension, injective dimension and flat dimension of $M$, and $\wgldim(R)$ and $\gldim(R)$ to denote, respectively, the weak and global homological dimensions of $R$.\\
Now, we review some definitions and notation. Let $J$ be an ideal of $R$.
Following \cite{HFX}, $J$ is called a \emph{Glaz-Vasconcelos ideal} (a $GV$-ideal for short) if $J$ is
finitely generated and the natural homomorphism $\varphi : R \rightarrow J^{\ast} = {\rm Hom}_R(J,R)$ is
an isomorphism. Note that the set $GV(R)$ of $GV$-ideals of $R$ is a multiplicative system of ideals of $R$. Let $M$ be an $R$-module. It is Defined
$${\rm tor}_{GV}(M) = \{x\in M \mid Jx = 0\text{  for some } J\in  GV(R)\}.$$
It is clear that $\tor_{GV}(M)$
is submodule of $M$. $M$ is said to be $GV$-torsion (resp., $GV$-torsion-free) if $\tor_{GV}(M) =M$ (resp., $ \tor_{GV}(M) =0$). A $GV$-torsion-free module $M$ is called a $w$-module if ${\rm Ext}^1_R(R/J, M) =0$ for any $J \in GV(R)$. Then, projective modules and reflexive modules are $w$-modules. In the recent paper \cite{HFX}, it was shown that flat modules are $w$-modules. Also it is known that a $GV$-torsion-free $R$-module $M$ is a $w$-module if and only $\Ext^R_1(N,M)=0$ for every
$GV$-torsion $R$-module $N$ (see, \cite{FH1}, Theorem 6.2.7). The notion of $w$-modules was introduced firstly over a domain \cite{FWM2} in the study of Strong Mori domains and was extended to commutative rings with zero divisors in \cite{HFX}. Let $w-Max(R)$ denote the set of $w$-ideals of $R$ maximal among proper integral $w$-ideals of $R$ (maximal $w$-ideals). Following  \cite[Proposition 3.8]{HFX}, every maximal $w$-ideal is prime.  For any
$GV$-torsion free module $M$,
$$ M_{w}:=\{x\in E(M)\mid Jx\subseteq M \text{  for some } J\in  GV(R)\} $$
is a $w$-submodule of $E(M)$ containing $M$ and is called the $w$-envelope of $M$, where $E(M)$ denotes the injective hull  of $M$. It is clear that a $GV$-torsion-free module $M$ is a $w$-module if and only if $M_{w}=M$.\\
Let $M$ and $N$ be $R$-modules and let $f : M \rightarrow N$ be a homomorphism. Following \cite{FW},
$f$ is called a $w$-monomorphism (resp., $w$-epimorphism, $w$-isomorphism) if $f_{\mathfrak{m}} :
M_{\mathfrak{m}}\rightarrow N_{\mathfrak{m}}$ is a monomorphism (resp., an epimorphism, an isomorphism) for all
$\mathfrak{m}\in w-Max(R)$. A sequence $A \rightarrow B  \rightarrow C$  of modules and homomorphisms is called $w$-exact if the sequence $A_{\mathfrak{m}} \rightarrow B_{\mathfrak{m}}  \rightarrow C_{\mathfrak{m}}$ is exact for all $\mathfrak{m}\in w-Max(R)$. An $R$-module $M$ is said to be of finite type   if there exists a finitely generated free $R$-module $F$ and a $w$-epimorphism $g : F\rightarrow M$. Similarly, an $R$-module $M$ is said to be of finitely presented type if there exists a $w$-exact sequence $F_1\rightarrow F_0 \rightarrow M \rightarrow 0$, where $F_1$ and $F_0$ are finitely generated free.\\
In recent years, homological theoretic characterization of $w$-modules has
received attention in several papers the literature (for example see [\cite{FMR}, \cite{FCZ}, \cite{FL}, \cite{FLQ}]). The notion of $w$-projective modules and $w$-flat modules appeared first in \cite{WAN} when $R$ is an integral domain and was extended to an arbitrary commutative ring in [\cite{FHT}, \cite{HF}]. In \cite{FHT}, F. G. Wang and H. Kim generalized projective modules to $w$-projective modules by the $w$-operation. An $R$-module $M$ is said to be a $w$-projevtive if $\Ext^1_R(L(M),N)$ is $GV$-torsion for any torsion-free $w$-module $N$, where $L(M)=(M/\tor_{GV}(M))_w$. Denote by  $\mathcal{P}_{w}$ the class of all $w$-projective $R$-modules. Following \cite{WLQ}, an $R$-module $M$ is a $w$-split if and only if  $\Ext_R^1(M,N)$ is $GV$-torsion for all $R$-modules $N$. Denote by $\mathcal{S}_{w}$ the class of all $w$-split $R$-modules. Hence, by [\cite{Wang and Lie}, Corollary 2.4], every $w$-split module is $w$-projective.
Following \cite{HF}, an $R$-module $M$ is said to be $w$-flat if for any $w$-monomorphism $f:A \rightarrow B$, the induced sequence $1\otimes f: M \otimes_R A \rightarrow M \otimes_R B$ is a $w$-monomorphism. Denote by $\mathcal{F}_{w}$ the class of all $w$-flat $R$-modules. Following \cite{FLQ}, throughout this paper, $\mathcal{P}_{w}^{\dagger_\infty}$ denote the class of $GV$-torsion-free $R$-modules $N$ with the property that $\Ext^k_R(M,N)=0$ for all $w$-projective $R$-modules $M$ and for all integers $k\geq1$ Clearly, every $GV$-torsionfree injective $R$-module belongs to $\mathcal{P}_{w}^{\dagger_\infty}$. An $R$-module $M$ is said to be weak $w$-projective if $\Ext^1_R(M, N)=0$ for all $N \in \mathcal{P}_{w}^{\dagger_\infty}$: Denote by w$\mathcal{P}_w$ the class of all weak $w$-projective modules. Following \cite{FLQ},  Wang and Qiao introduce the notions of the weak $w$-projective dimension ($\wwpd$) of a module and the global weak $w$-projective dimension ($\glwdim$) of a ring.
 Following \cite{FLQ}, a $GV$-torsion-free module $M$ is said to be a strong w-module if $\Ext^i_R(N,M)=0$for any integer $i \geq1$ and all $GV$-torsion modules $N$. Denote by $\mathcal{W}_{\infty}$ the class of all strong $w$-modules. Then all $GV$-torsion-free injective modules are strong $w$-modules. Clearly, $\mathcal{P}_{w}^{\dagger_\infty}\subseteq \mathcal{W}_{\infty}$. But, in \cite{FLQ}, they do not showed that $\mathcal{P}_{w}^{\dagger_\infty}$ and $\mathcal{W}_{\infty}$ are the different class of $R$-modules, and this question was answered in \cite{Pu wang}.\\
 Recall from \cite{LMND} that an $R$-module $M$ is called FP-projective if $\Ext^1_R(M,N)=0$ for any absolutely pure $R$-module $N$. Denote by $\mathcal{FP}$ the class of all FP-projective modules.
Recall that an $R$-module $A$ is an absolutely pure if $A$ is a pure submodule in every $R$-module which contains $A$ as a submodule (see, \cite{BHM}).  C. Megibbeni showed in \cite{CM}, that an $R$-module $A$ is absolutely pure if and only if $\Ext_R^1(F,A)=0$, for every finitely presented module $F$. Hence, an absolutely pure module is precisely a FP-injective module in \cite{BS}.
	\section{results}
	In this section, we introduce a characterize of some classical ring. But we need the following lemma
\begin{lemma}(\cite{FLQ}, Proposition 2.5)\label{lemma}
An $R$-module $M$ is weak $w$-projective if $\Ext_R^1(M,N)=0$ for all  $N \in \mathcal{P}_{w}^{\dagger_\infty}$ and for all $k\geq0$.
\end{lemma}

 It is obvious that, for the class of modules\\
$\left\{\text{ projective }\right\}\subseteq \left\{\text{ $w$-split }\right\}\subseteq\left\{\text{ $w$-proective }\right\} \subseteq \left\{\text{weak $w$-proective }\right\} \subseteq \left\{\text{ $w$-flat }\right\}$.\\
By [\cite{FMR}, Proposition 2.5], if $R$ is a perfect ring, then the five classes of modules above coincide.

In the following proposition, we will give some characterizations of weak $w$-projective modules.
	\begin{proposition}\label{prop1}
		Let $M$ be an $R$-module. Then the following are equivalent:
		\begin{enumerate}
			\item $M$ is  weak $w$-projective.
			\item $M\otimes F$ is  weak $w$-projective for any projective $R$-module $F$.
			\item $\Hom_R(F,M)$ is  weak $w$-projective for any finitely generated projective $R$-module $F$.
			\item For any exact sequence of $R$-modules
			$\xymatrix
			{ 0\rightarrow A\rightarrow B\rightarrow C\rightarrow 0 }
			$\ with $A\in\mathcal{P}_{w}^{\dagger_\infty}$, the sequence
			$\xymatrix
			{ 0\rightarrow \Hom_R(M,A)\rightarrow \Hom_R(M,B)\rightarrow \Hom_R(M,C)\rightarrow 0}$ \ is exact.
			\item For any $w$-exact sequence of $R$-modules
			$\xymatrix
			{ 0\rightarrow L\rightarrow E\rightarrow M\rightarrow 0 }
			$\, the sequence
			$\xymatrix
			{ 0\rightarrow \Hom_R(M,N)\rightarrow \Hom_R(E,N)\rightarrow \Hom_R(L,N)\rightarrow 0}$ \ is exact for any $R$-module $N\in\mathcal{P}_{w}^{\dagger_\infty}$.
			\item For any exact sequence of $R$-modules
			$\xymatrix
			{ 0\rightarrow L\rightarrow E\rightarrow M\rightarrow 0 }
			$\, the sequence
			$\xymatrix
			{ 0\rightarrow \Hom_R(M,N)\rightarrow \Hom_R(E,N)\rightarrow \Hom_R(L,N)\rightarrow 0}$ \ is exact for any $R$-module $N\in\mathcal{P}_{w}^{\dagger_\infty}$.
		\end{enumerate}
	\end{proposition}
	\begin{proof}
	$(1)\Rightarrow (2)$. Let $F$ be a projective $R$-module. For any $R$-module $N$ in $\PP_{w}^{\dagger_\infty}$, we have $\Ext_R^1(F\otimes M,N)\cong \Hom_R(F,\Ext_R^1(M,N))$ by [\cite{FH1}, Theorem 3.3.10]. Since $M$ is a weak $w$-projective, $\Ext_R^1(M,N)=0$. Thus, $\Ext_R^1(F\otimes M,N)=0$. Hence, $F\otimes M$ is a weak $w$-projective.\\
	$(2)\Rightarrow (1)$ and $(3)\Rightarrow (1)$. Follow by letting $F=R$.\\
	$(1)\Rightarrow (3)$. Let $N\in \mathcal{P}_{w}^{\dagger_\infty}$, for any finitely generated projective $R$-module $F$, we have $F\otimes \Ext_R^1(M,N)\cong \Ext_R^1(\Hom_R(F,M),N)$ by  [\cite{FH1}, Theorem 3.3.12]. Since $M$ is weak $w$-projective, so $\Ext_R^1(M,N)=0$. Hence, $\Ext_R^1(\Hom_R(F,M),N)=0$, which implies that $\Hom_R(F,M)$ is weak $w$-projective.\\
	$(1)\Rightarrow (4)$. Let
	$ 0\rightarrow A\rightarrow B\rightarrow C\rightarrow 0$ \ be an exact sequence with $A\in\mathcal{P}_{w}^{\dagger_\infty}$, then we have the exact sequence
	$0\rightarrow \Hom_R(M,A)\rightarrow \Hom_R(M,B)\rightarrow \Hom_R(M,C)\rightarrow \Ext_R^1(M,A)$. Since $M$ is weak $w$-projective and $A\in\mathcal{P}_{w}^{\dagger_\infty}$, so $\Ext_R^1(M,A)=0$. Thus, $0\rightarrow \Hom_R(M,A)\rightarrow \Hom_R(M,B)\rightarrow \Hom_R(M,C)\rightarrow 0$ is exact.\\
	$(4)\Rightarrow (1)$. Let $N\in\mathcal{P}_{w}^{\dagger_\infty}$, consdier an exact sequence
	  $ 0\rightarrow N\rightarrow E\rightarrow L\rightarrow 0$
	with $E$ is injective module, then we have the exact sequence
	$0\rightarrow \Hom_R(M,N)\rightarrow \Hom_R(M,E)\rightarrow \Hom_R(M,L)\rightarrow \Ext_R^1(M,N)\rightarrow 0$, and keeping in mind that $0\rightarrow \Hom_R(M,N)\rightarrow \Hom_R(M,E)\rightarrow \Hom_R(M,L)\rightarrow 0$ is exact, we deduce that $\Ext_R^1(M,N)=0$. Hence, $M$ is weak $w$-projective.\\
	$(1)\Rightarrow(5)$. Let
	$ 0\rightarrow L\rightarrow E\rightarrow M\rightarrow 0$ \ be a  $w$-exact sequence. For any $R$-module $N\in\mathcal{P}_{w}^{\dagger_\infty}$ so $N\in \mathcal{W}_{\infty}$. By [\cite{FLQ}, Lemma 2.1], we have the exact sequence $0\rightarrow \Hom_R(M,N)\rightarrow \Hom_R(E,N)\rightarrow \Hom_R(L,N)\rightarrow \Ext_R^1(M,N)$. Since $M$ is weak $w$-projective, so $\Ext_R^1(M,N)=0$, and $(5)$ is holds.\\
	$(5)\Rightarrow(6)$. Trivial.\\
	$(6)\Rightarrow(1)$. Let
	$ 0\rightarrow L\rightarrow E\rightarrow M\rightarrow 0$ be an exact sequence with $E$ is projective. Hence, for any $R$-module $N\in\mathcal{P}_{w}^{\dagger_\infty}$, we have
	$0\rightarrow \Hom_R(M,N)\rightarrow \Hom_R(E,N)\rightarrow \Hom_R(L,N)\rightarrow \Ext_R^1(M,N)\rightarrow 0$ is exact sequence,
	    and keeping in mind that  $0\rightarrow \Hom_R(M,N)\rightarrow \Hom_R(E,N)\rightarrow \Hom_R(L,N)\rightarrow 0$ is exact, we deduce that $\Ext_R^1(M,N)=0$, which implies that $M$ is weak $w$-projective.
\end{proof}
Recall from \cite{FW}, that a ring is said to be $w$-coherent if every finitely generated ideal of $R$ is of finitely presented type.
\begin{proposition}\label{propp}
	Let $R$ be a $w$-coherent ring, $E$ be an injective $R$-module, $M$ be a finitely presented type and $N$ be an $R\{x\}$-module. Then, if $M$ is weak $w$-projective $R$-module, so $\Tor_n^R(M,\Hom(N,E))=0$.
\end{proposition}
\begin{proof} Let $M$ be a weak $w$-projective $R$-module and let $N$ be an $R\{x\}$-modul, so $\Ext_R^n (M,n)=0$ by [\cite{FLQ}, Proposition 2.5] and since every $R\{x\}$-module in $\mathcal{P}_{w}^{\dagger_\infty}$ by [\cite{FLQ}, Proposition 2.4]. Henace, by [\cite{wang and Kim}, Proposition 2.13(6)], we have $$\Tor^R_n(M,\Hom(N,E))\cong \Hom(\Ext^n_R(M,N),E)=0.$$ Which implies that,  $\Tor^R_n(M,\Hom(N,E))=0$
\end{proof}
\begin{proposition}\label{lema4}
	Every weak $w$-projective of finite type is of finitely presented type.
\end{proposition}	
\begin{proof} Let $M$ be a weak $w$-projective $R$-module of finite type, so by [\cite{FLQ}, Corollary 2.9] $M$ is $w$-projective of finite type. Thus, by [\cite{FH1}, Theorem 6.7.22], we have $M$ is finitely presented type.
\end{proof}

\begin{proposition}\label{propo5}
	Let $M$ be a $GV$-torsion-free module. The following assertions hold.
	\begin{enumerate}
		\item $M_w/M$ is a weak $w$-projective module.
		\item $M$ is a weak $w$-projective if and only if  so is $M_w$.
	\end{enumerate}
\end{proposition}
\begin{proof}
	$(1)$. Let $M$ be a $GV$-torsion-free module. So, by [\cite{FH1}, Proposition 6.2.5] we have $M_w/M$ a $GV$-torsion module. Hence, by [\cite{FLQ}, Proposition 2.3(2)], we have $M_w/M$ is weak $w$-projective.\\
	$(2)$. Let $N$ be an $R$-module in $\mathcal{P}_{w}^{\dagger_\infty}$. Since $M$ is $GV$-torsion-free, we have by $(1)$ $M_w/M$ is weak $w$-projective module. Consider the following exact sequence
	$$ 0\rightarrow M\rightarrow M_w\rightarrow M_w/M\rightarrow 0$$ which is $w$-exact. Hence, by [\cite{FLQ}, Proposition 2.5], $M$ is  weak $w$-projective if and only if $M_w$ is weak $w$-projective.
\end{proof}
Recall that a ring $R$ is called a $DW$-ring if every ideal of $R$ is a $w$-ideal, or equivalently every maximal ideal of $R$ is $w$-ideal \cite{AM}. Examples of $DW$-rings are Pr\"{u}fer domains, domains with Krull dimension one, and rings with Krull dimension zero. We note that if $R$ is $DW$-ring, then every  $R$-module  in $\mathcal{P}_{w}^{\dagger_\infty}$.\\
In the following proposition, we will give a new characterizations of $DW$-rings which are the only rings with these properties.

\begin{proposition}\label{prrop3}
	Let $R$ be a ring. The following statements are equivalent:
	\begin{enumerate}
		\item Every weak $w$-projective $R$-module is projective.
		\item Every $w$-projective $R$-module is projective.
		\item Every $GV$-torsion $R$-module is projective.
		\item Every $GV$-torsion-free $R$-module is strong $w$-module.
		\item Every finitely presented type $w$-flat is projective.
		\item Every weak $w$-projective $R$-module is $w$-module.
		\item $R$ is $DW$-ring.
	\end{enumerate}
\end{proposition}
\begin{proof}
	$(1)\Rightarrow (2) \;and\; (2)\Rightarrow (3)$. The are trivial.\\
	$(3)\Rightarrow (4)$. Let $M$ be a $GV$-torsion-free $R$-module, for any $GV$-torsion $R$-module $N$, we have $\Ext^i_R(N,M)=0$ since $N$ is projective. Hence, $M$ is a strong $w$-module.\\
	$(4)\Rightarrow (7)$. By [\cite{FW}, Theorem 3.8] since every strong $w$-module is $w$-module.\\
	$(1)\Rightarrow (6)$. Trivial, since every projective $R$-module is $w$-module.\\
	$(2)\Rightarrow (5)$. Let $M$ be a finitely presented type $w$-flat. By [\cite{FLQ}, Corollary 2.9], we have $M$ is a finite type $w$-projective. Hence, $M$ is a projective $R$-module by $(2)$.\\
	$(5)\Rightarrow (7)$. Let $M$ be a finitely presented $w$-flat. Then,  $M$ is finitely presented type $w$-flat, so $M$ is projective by $(5)$. Hence, by [\cite{MRM}, Proposition 2.1], we have $R$ is a $DW$-ring.\\
	$(6)\Rightarrow (7)$. Let $M$ be a $GV$-torsion-free $R$-module. Hence, by Proposition \ref{propo5}, we have $M_w/M$ is weak $w$-projectiveis and so $w$-module by $(6)$. Thus, $M_w/M$ is a $GV$-torsion-free. Hence, $M_w/M=0$ and so $M_w=M$. Thus, $M$ is $w$-module. Then, $R$ is a $DW$-ring by [\cite{FW}, Theorem 3.8].\\
	$(7)\Rightarrow (1)$. Let $M$ be a weak $w$-projective. For any $R$-module $N$, we have $\Ext^1_R(M,N)=0$ because $N\in\mathcal{P}_{w}^{\dagger_\infty}$ (since $R$ is $DW$). Hence, $M$ is a projective module.
\end{proof}
Note that the equivalence $(1)\Leftrightarrow (7)$ in Proposition \ref{prrop3} was given in [\cite{Pu wang}, Proposition 4.4] for the domain case.

L. Mao and N. Ding in [\cite{LMND}], proved that a ring $R$ is a Von Neumann regular if and only if every FP-projective $R$-module is projective.\\
Next, we will give new characterizations of a Von Neumann regular rings by weak $w$-projective modules.	
\begin{proposition}\label{prrop1}
	Let $R$ be a ring. Then, the following statements are equivalent:
	\begin{enumerate}
		\item Every FP-projective $R$-module is weak $w$-projective.
		\item Every finitely presented $R$-module is weak $w$-projectiv.
		\item Every finitely presented $R$-module is $w$-flat.
		\item $R$ is a Von Neumann regular.
	\end{enumerate}
\end{proposition}
\begin{proof} $(1)\Rightarrow(2)$. Follows from the fact that every finitely
	presented $R$-module is FP-projective.\\
	$(2)\Rightarrow(3)$. Let $M$ be a finitely presented $R$-module, so $M$ is weak $w$-projective. Hence, $M$ is $w$-flat by [\cite{FLQ}, Corollary 2.11].\\
	$(3)\Rightarrow(4)$. Let $I$ be a finitely generated ideal of $R$, then $R/I$ is finitely presented. So $R/I$ is $w$-flat by $(3)$, then $\wfd_R(R/I)=0$. Thus, $\w-wgldim(R)=0$ by [\cite{FL}, Proposition 3.3]. Hence, $R$ is Von Neumann regular by [\cite{FH}, Theorem 4.4].\\
	$(4)\Rightarrow (1)$. Let $M$ be a FP-projective, so $M$ is projective by [\cite{LMND}, Remarks 2.2]. Hence, $M$ is a weak $w$-projective.
\end{proof}
Next, we will give an example of FP-projective module which is not weak $w$-projective.
\begin{example}\label{example}
	Consider the local Quasi-Frobenius ring $R:=k[X]/(X^2)$ where $k$ is a field, and denote by $\overline{X}$ the residue class in $R$ of $X$. Then, $(\overline{X})$ is FP-projective $R$-module which is not weak $w$-projective.	
\end{example}
\begin{proof} Since $R$ is a Quasi-Frobenius ring, then every absolutely pure $R$-module is injective. Hence, for any absolutely pure $R$-module $N$, we have $\Ext_R^1((\overline{X}),N)=0$, so $(\overline{X})$ is FP-projective. But, $(\overline{X})$ is not projective by [\cite{TCL}, Example 2.2], and so not weak $w$-projective, since $R$ is $DW$-ring.
\end{proof}
Recall from [\cite{FH}] that a ring $R$ is said to be $w$-semi-hereditary if every finite type ideal of $R$ is $w$-projective.
\begin{proposition}\label{prop5}
	The following are equivalent:
	\begin{enumerate}
		\item $R$ $w$-semi-hereditary.
		\item Every finite type submodule of a free module is weak $w$-projective.
		\item Every finite type ideal of $R$ is a weak $w$-projective.
		\item Every finitely generated submodule of a free module is weak $w$-projective.
		\item Every finitely generated ideal of $R$ is a weak $w$-projective
	\end{enumerate}
\end{proposition}
\begin{proof} $(1)\Rightarrow (2)$. Let $J$ be a finite type submodule of a any free module. Hence, $J$ is $w$-projective by [\cite{FH}, Theorem 4.11]. Then $J$ is weak $w$-projective by [\cite{FLQ}, Corollary 2.9].\\
	$(2)\Rightarrow (3)\Rightarrow (5)$ and $(2)\Rightarrow (4)\Rightarrow (5)$. These are trivial.\\
	$(5)\Rightarrow (1)$. Let $J$ be a finite type ideal of $R$. Then $J$ is $w$-isomorphic to a	finitely generated subideal $I$ of $J$. Hence $J$ is weak $w$-projective by hypothesis and
	[\cite{FLQ}, Corollary 2.7].
\end{proof}

\begin{proposition}\label{propo}
	Every $GV$-torsion-free weak $w$-projective module is torsion-free.
\end{proposition}
\begin{proof}
	Let $M$ be a $GV$-torsion-free weak $w$-projective module. Hence, $M$ is a $GV$-torsion-free $w$-flat by [\cite{FLQ}, Corollary 2.11].Thus, by [\cite{FH1}, Proposition 6.7.6], we have $M$ is torsion-free.
\end{proof}
In the next example we will prove that a weak $w$-projective module need not to be torsion-free.

\begin{example}\label{exam}
	Let $R$ be an integral domain and $J$ be a proper $GV$-ideal of $R$. Then $M:=R \oplus R/J$
	is a weak $w$-projective module but not torsion-free.
\end{example}
\begin{proposition}\label{prop3}
	Let $R$ be a ring and $M$ be a finitely presented $R$-module. Then, the following statements are equivalent:
	\begin{enumerate}
		\item $M$ is $w$-split.
		\item $M$ is weak $w$-projective.
		\item For any $w$-exact
		$\xymatrix
		{ 0\rightarrow A\rightarrow B\rightarrow C\rightarrow 0 }$, the sequence\\
		$\xymatrix
		{ 0\rightarrow \Hom_R(M,A)\rightarrow \Hom_R(M,B)\rightarrow \Hom_R(M,C)\rightarrow 0}$ is $w$-exact.
	\end{enumerate}
\end{proposition}
\begin{proof} $(1)\Rightarrow (2)$. Trivial, since every $w$-split $R$-module is weak $w$-projective.\\
	$(2)\Rightarrow (3)$. Let $ 0\rightarrow A\rightarrow B\rightarrow C\rightarrow 0$ be a $w$-exact sequence of $R$-modules. Then, for any maximal $w$-ideal $\m$ of $R$, $ 0\rightarrow A_{\m}\rightarrow B_{\m}\rightarrow C_{\m}\rightarrow 0$ is exact sequence of $R_{\m}$-modules. Thus, since $M_{\m}$ is free by [\cite{FLQ}, Proposition 2.8], we have the exact sequence
	$\xymatrix
	{ 0\rightarrow \Hom_R(M_{\m},A_{\m})\rightarrow \Hom_R(M_{\m},B_{\m})\rightarrow \Hom_R(M_{\m},C_{\m})\rightarrow 0}$.
	Since $M$ is finitely presented, we have the commutative diagram
	$$\begin{array}{ccccc}
	{\rm Hom}_{R_{\mathfrak{m}}}(M_{\mathfrak{m}}, A_{\mathfrak{m}}) & \rightarrow & {\rm Hom}_{R_{\mathfrak{m}}}(M_{\mathfrak{m}},B_{\mathfrak{m}} ) & \rightarrow & {\rm Hom}_{R_{\mathfrak{m}}}(M_{\mathfrak{m}},C_{\mathfrak{m}}) \\
	\mid\mid\wr &  & \mid\mid\wr &  & \mid\mid\wr \\
	{\rm Hom}_{R}(M, A)_{\mathfrak{m}} & \rightarrow & {\rm Hom}_{R}(M,B )_{\mathfrak{m}} & \rightarrow & {\rm Hom}_{R}(M,C)_{\mathfrak{m}}
	\end{array}$$
	Thus, $0\rightarrow {\rm Hom}_{R}(M, A)_{\mathfrak{m}}  \rightarrow {\rm Hom}_{R}(M,B )_{\mathfrak{m}}   \rightarrow   {\rm Hom}_{R}(M,C)_{\mathfrak{m}}\rightarrow 0$ is exact, and so, $0\rightarrow {\rm Hom}_R(M,A)\rightarrow {\rm Hom}_R(M,B) \rightarrow {\rm Hom}_R(M,C)\rightarrow 0$ is $w$-exact.\\
	$(3)\Rightarrow (1)$. By [\cite{WLQ}, Proposition 2.4].
\end{proof}
Recall from \cite{SXFW1}, that a $w$-exact sequence of $R$-modules $\xymatrix{0\rightarrow A\rightarrow B\rightarrow C\rightarrow 0}$ is said to be $w$-pure exact if, for any $R$-module $M$,
the induced sequence $$\xymatrix{0\rightarrow A\otimes M\rightarrow B\otimes M \rightarrow C\otimes M \rightarrow 0}$$ is $w$-exact.
\begin{proposition}\label{prop}
	Let $C$ be a finitely presented type $R$-module. Then, the following statements are equivalent:
	\begin{enumerate}
	\item $C$ is a weak $w$-projective $R$-module.
	\item Every $w$-exact sequence of $R$-modules $ 0\rightarrow A\rightarrow B\rightarrow C\rightarrow 0$ is $w$-pure exact.
\end{enumerate}
\end{proposition}
\begin{proof} $(1)\Rightarrow(2)$. Since every weak $w$-projective is $w$-flat by [\cite{FLQ}, Corollary 2.11]. Hence, by [\cite{SXFW1}, Theorem 2.6], we have the result.\\
	$(2)\Rightarrow(1)$. Let $ 0\rightarrow A\rightarrow B\rightarrow C\rightarrow 0$ be a $w$-exact sequence, so is a $w$-pure exact by hypothesis. Thus, $C$ is $w$-flat by [\cite{SXFW1}, Theorem 2.6]. Hence $C$ is a weak $w$-projective by [\cite{FLQ}, Corollary 2.9].
\end{proof}

\begin{proposition}\label{pprop}
	The following are equivalent for a finite type $R$-module $M$.
	\begin{enumerate}
		\item $M$ is a $w$-projective module.
		\item $\Ext_R^1(M,B)=0$ for any $B \in \mathcal{P}_{w}^{\dagger_\infty}$.
		\item $\Ext_R^1(M,N)=0$ for any $R\{x\}$-module $N$.
		\item $M\{x\}$ is a projective $R\{x\}$-module.
	\end{enumerate}
\end{proposition}
\begin{proof} $(1)\Rightarrow (2)$. This is trivial.\\
$(2)\Rightarrow (3)$. By [\cite{FLQ}, Proposition 2.4].\\
$(3)\Rightarrow (4)$. Let $N$ be an $R\{x\}$-module, we have by [\cite{wang and Kim}, Proposition 2.5],
$$\Ext^n_{R\{x\}}(M\{x\},N)\cong \Ext^n_{R}(M,N)=0.$$ Thus, $M\{x\}$ is a projective $R\{x\}$-module.\\
$(4)\Rightarrow (1)$. Let $M\{x\}$ be a projective $R\{x\}$-module, so $M\{x\}$ is finitely generated by [\cite{FH1}, Theorem 6.6.24] and since $M$ is of finite type. Hence, by [\cite{FH1}, Theorem 6.7.18], $M$ is $w$-projective module.
\end{proof}

Recall form [\cite{FLQ}], that an $R$-module $D$ is said to be $\mathcal{P}_{w}^{\dagger_\infty}$-divisible if it is isomorphic to $E/N$ where $E$ is a $GV$-torsin-free injective $R$-module and $N \in \mathcal{P}_{w}^{\dagger_\infty}$ is a submodule of $E$.
\begin{proposition}\label{h-divi}
	Let $M$ be an $R$-module and any integer $m \geq1$. The following are equivalent.
	\begin{enumerate}
		\item $\wwpd_R M\leq m$.
		\item $\Ext_R^m(M,D)=0$ for all $\mathcal{P}_{w}^{\dagger_\infty}$-divisible $R$-module $D$.
	\end{enumerate}
\end{proposition}
\begin{proof}
$(1)\Rightarrow (2)$. Let $N\in \mathcal{P}_{w}^{\dagger_\infty}$. Then there exists an exact sequence of $R$-modules  $0 \rightarrow N \rightarrow E \rightarrow H \rightarrow 0$, where $E$ is a $GV$-torsion-free injective $R$-module. Hence, $D$ is $\mathcal{P}_{w}^{\dagger_\infty}$-divisibl. Then we have the induced exact sequence
$$\Ext^m_R(M, H) \rightarrow \Ext^{m+1}_R(M,N) \rightarrow \Ext^{m+1}_R(M,E)=0,$$ for any integer $m \geq1$. The left term is zero by hypothesis. Hence, $\Ext^{m+1}_R(M,N)=0$, which implies that $\wwpd_R M\leq m$ by [\cite{FLQ}, Proposition 3.1].\\
$(2)\Rightarrow (1)$. Let $\wwpd_R M\leq m$ and $D$ be a $\mathcal{P}_{w}^{\dagger_\infty}$-divisible $R$-module. Then we have an exact sequence $0 \rightarrow N \rightarrow E \rightarrow H \rightarrow 0$, where $E$ is a $GV$-torsion-free injective $R$-module and $N\in \mathcal{P}_{w}^{\dagger_\infty}$. Hence, we have the exact sequence
$$0=\Ext^m_R(M,E) \rightarrow \Ext^{m}_R(M,H) \rightarrow \Ext^{m+1}_R(M,N).$$ The right term is zero by [\cite{FLQ}, Proposition 3.1]. Therefore, $\Ext^{m}_R(M,H)=0$.
\end{proof}

\begin{proposition} Let $M$ and $N$ be two $R$-modules. Then,
\begin{center}
$\wwpd_R(M\oplus N)=\sup\{\wwpd_R M,\wwpd_R N\}$
	\end{center}
\end{proposition}
\begin{proof} The inequality $\wwpd_R(M\oplus N)\leq\sup\{\wwpd_R M,\wwpd_R N\}$  follows from the fact that the class of weak $w$-projective modules is closed under direct sums by [\cite{FLQ}, Proposition 2.5(1)]. For the converse inequality, we may assume
	that $\wwpd_R(M\oplus N)= n$ is finite. Thus, for any $R$-module $X\in\mathcal{P}_{w}^{\dagger_\infty}$,
	$${\rm Ext}^{n+1}_R(M\oplus N,X)\cong {\rm Ext}^{n+1}_R(M,X)\oplus {\rm Ext}^{n+1}_R( N,X).$$ Since ${\rm Ext}^{n+1}_R(M\oplus N,X)=0$ by [\cite{FLQ}, Proposition 3.1]. Hence, ${\rm Ext}^{n+1}_R(M,X)={\rm Ext}^{n+1}_R(N,X)=0$, which implies that, $\sup\{\wwpd_R M,\wwpd_R N\}\leq n$.
\end{proof}


\begin{thebibliography}{999}\addcontentsline{toc}{section}{\protect\numberline{}{Bibliography}}
		\par\bibitem{FMR} F. A. Almahdi, M. Tamekkante and R. A. K. Assaad, On the right orthogonal complement of the class of w-flat modules, J. Ramanujan Math. Soc. {\bf33} No.2 (2018) 159–175.
		\par\bibitem{HF} H. Kim and F. Wang, On LCM-stable modules, J. Algebra Appl. 13, no. 4 (2014), 1350133, 18 pages.
		\par\bibitem{BHM} B. H. Maddox, Absolutely pure modules, Proc. Amer. Math. Soc. 18 (1967) 155–158.
		\par\bibitem{LMND} L. Mao and N. Ding, FP-projective dimension, Comm. in Algebra. {\bf33} (2005) 1153--1170.
		\par\bibitem{CM} C. Megibben, Absolutely pure modules, Proc. Am. Math. Soc. 26 (1970) 561-566.
		\par\bibitem{AM}  A. Mimouni,   Integral domains in which each ideal is a $w$-ideal, Commun. Algebra {\bf 33} (2005), 1345--1355.
		\par\bibitem{BS} B. Stenstr\"{o}m, Coherent rings and FP-injective modules, J. Lond. Math. Soc. 2(2) (1970) 323–329.
		\par\bibitem{Pu wang} Y. Y. Pu, W. Zhao, G. H. Tang, and F. G. Wang, $w_\infty$-projective modules and Krull domains, Commun. Algebra, Vol. 50, No. 8, (2022), 3390–3402.
		\par\bibitem{MRM} M. Tamekkante, R. A. K. Assaad and
		E. Bouba, Note  On The  DW Rings, Inter. Elec. J. of Algebra. VO. {\bf25} (2019).
		\par\bibitem{TCL} M. Tamekkante, M. Chhiti and K.Louartiti, Weak Projective Modules and Dimension, Int. J. of Algebra. {\bf5} (2011) 1219 -1224.		
		\par\bibitem{WAN} F. Wang, On $w$-projective modules and $w$-flat modules, Algebra Colloq. {\bf4} (1997), no. 1, 111-120.
		\par\bibitem{FW} F. Wang, Finitely presented type modules and w-coherent rings, J. Sichuan Normal Univ. 33 (2010) 1–9.
		\par\bibitem{FH1} F. Wang and H. Kim, Foundations of Commutative Rings and Their Modules, (Springer Nature Singapore Pte Ltd., Singapore, 2016).
		\par\bibitem{FHT} F. Wang and H. Kim, Two generalizations of projective modules and their applications, J. Pure Applied Algebra 219 (2015) 2099-2123.
		\par\bibitem{FH} F. Wang and H. Kim, $w$-injective modules and $w$-semi-hereditary rings, J. Korean Math. Soc. 51 (2014), no. 3, 509–525.
		\par\bibitem{wang and Kim} F. Wang and H. Kim, Relative FP-injective modules and relative IF rings, Commun. Algebra, Vol. 49, (2021), 3552-3582.
		\par\bibitem{FWM2} F. Wang and R. L. McCasland, On $w$-modules over strong Mori domains, Comm. Algebra 25(4), 1285-1306 (1997).
       \par\bibitem{FLQ} F. Wang and L. Qiao, A homological characterization of Krull domains II, Comm. in Algebra. (2019).
       \par\bibitem{FL} F. Wang and L. Qiao, The $w$-weak global dimension of commutative rings, Bull. Korean Math. Soc. 52 (2015), no. 4, 1327–1338.
       \par\bibitem{WLQ} F. Wang and L. Qiao,  A new version of a theorem of Kaplansky. arXiv: 1901.02316.
       \par\bibitem{FCZ} F. G. Wang and D. C. Zhou, A homological characterization of Krull domains, Bull. Korean Math. Soc. 55 (2018), no. 2, 649–657.
       	\par\bibitem {SXFW1} S. Xing and F. Wang, {\em Purity over Pr\"{u}fer v-multiplication domains}, J. of Algebra Appl. Vol. \textbf{16}, No. 5 1850100 (2018).
    	\par\bibitem{HFX} H. Y. Yin, F. G. Wang, X. S. Zhu and Y. H. Chen, $w$-modules over commutative rings, J. Korean. Math. Soc. 48(1) (2011) 207–222.   		
	\end{thebibliography}
\end{document}